\documentclass{amsart}
\usepackage{amsmath, amssymb, mathrsfs, verbatim, multirow}

\theoremstyle{definition}
\newtheorem{Teo}{Theorem}[section]
\newtheorem{Def}[Teo]{Definition}
\newtheorem{Prop}[Teo]{Proposition}
\newtheorem{Obs}[Teo]{Remark}
\newtheorem{Lema}[Teo]{Lemma}
\newtheorem{Cor}[Teo]{Corollary}

\newtheorem{Exa}[Teo]{Example}
\newtheorem{Claim}[Teo]{Claim}

\newcommand{\R}{\mathbb{R}}
\newcommand{\Z}{\mathbb{Z}}
\newcommand{\N}{\mathbb{N}}
\newcommand{\C}{\mathbb{C}}
\newcommand{\p}{\mathbb{P}}

\newcommand{\Llr}{\Longleftrightarrow}
\newcommand{\lra}{\longrightarrow}

\newcommand{\VR}{\mathcal{O}}
\newcommand{\PI}{\mathfrak{p}}
\newcommand{\MI}{\mathfrak{m}}

\begin{document}
\title[Infima and Topologies]{Valuations centered at a two-dimensional regular local ring: Infima and Topologies}

\date{}

\author{Josnei Novacoski}
\address{Department of Mathematics and Statistics,
University of Saskatchewan,
Saskatoon, \newline \indent
SK S7N 5E6, Canada}
\email{jan328@mail.usask.ca}


\begin{abstract} 
The aim of this paper is to prove that every non-empty set of valuations centered at a two-dimensional regular domain has an infimum. We also generalize some results related to a non-metric tree. 
\end{abstract}

\maketitle

\section{Introduction}

Favre and Jonsson prove in \cite{Fav_1} that the set of normalized valuations centered at $\C[[x,y]]$ has a tree structure. In \cite{Gra} Granja generalizes this result for the set of normalized valuations centered at any two-dimensional regular local ring. In both works, the definition of rooted non-metric tree is not satisfactory (see discussion after definition \ref{Def_2}). That is because the definition given in the cited papers does not guarantee the existence of infimum of a non-empty set of valuations. The existence of the infimum (of two valuations) is necessary in order to define some important concepts, such as the weak tree topology (see definition \ref{Def_1} item \textbf{(iv)}) and the metric associated to a parametrization (definition \ref{Def_1} item \textbf{(vi)}).

In \cite{Fav_1} it is stated that the existence of an infimum is a consequence of the given definition, which is not the case (see example \ref{Exa_1}). In order to make the theory developed there consistent, one needs to prove that there exists an infimum for any given pair of valuations. In the case of $R=\C[[x,y]]$ it is proved in \cite{Fav_1}, using the sequence of key polynomials associated to a valuation, that the infimum of two valuations exists as long as we can find an element which ``minimizes" both valuations. An easy argument (Corollary \ref{Cor_1}) shows that one always can get such an element (in fact, we can get it for any two-dimensional regular ring).

An interesting question which arises naturally is whether we can find an infimum of any non-empty set of valuations centered at a two-dimensional regular local ring. One of the purposes of this paper is to give a positive answer for this question.

\begin{Teo}\label{Main}
Let $R$ be a two-dimensional regular local ring and take any non-empty $\mathcal S=\{\nu_i\}_{i\in I}$ of centered valuations $\nu_i:R\lra \R_\infty$ normalized by $\nu_i(\MI)=1$ (see definition \ref{Norval}). Then there exists a valuation $\nu:R\lra \R_\infty$ which is the infimum of $\mathcal S$ with respect to the order given by $\nu\leq \mu$ if and only if $\nu(\phi)\leq \mu(\phi)$ for every $\phi\in R$.
\end{Teo}

By use of this theorem, it follows from \cite{Fav_1} and \cite{Gra} that the set of all centered normalized valuations on $R$ has a tree structure and associated to that a weak tree topology (see definitions and discussions in section \ref{Tree}). Every parametrization (see definition \ref{Def_1}) of a tree induces a metric on that tree. A natural question is whether this metric topology and the weak tree topology are comparable. The next theorem answers positively this question.

\begin{Teo}\label{Comp}
Let $(\mathcal T,\leq)$ be a rooted non-metric tree and let $\Psi:\mathcal T\lra [1,\infty]$ be a parametrization of $\mathcal T$. The weak tree topology on $\mathcal{T}$ is coarser than or equal to the topology associated with the metric $d_\Psi$.
\end{Teo}

In \cite{Fav_1}, two parametrizations (skewness and thinness) are presented to prove that the normalized centered valuations on $\C[[x,y]]$ form a parametrized non-metric tree. Also, in \cite{Gra}, Granja presents a new approach and a different parametrization that proves the same result for any two-dimensional regular local ring. In \cite{Fav_1}, Favre and Jonsson compare the topologies generated by their parametrizations and the weak tree topology. Our theorem above gives a more general comparison, which does not depend on the valuative origin of such tree.

We also show (Theorem \ref{Teo_1}) that if the tree has a point with uncountably many branches, then these topologies are different.

\section{Preliminaries}

\begin{Def}
Take a commutative ring $R$ with unity. A \textbf{valuation} on $R$ is a non-constant mapping $\nu:R\lra \Gamma_\infty :=\Gamma \cup\{\infty\}$ where $\Gamma$ is an ordered abelian group (and the extension of addition and ordering to $\infty$ is as usual), with the following properties:
\begin{description}
\item[(V1)] $\nu(\phi\psi)=\nu(\phi)+\nu(\psi)$ for all $\phi,\psi\in R$.
\item[(V2)] $\nu(\phi+\psi)\geq \min\{\nu(\phi),\nu(\psi)\}$ for all $\phi,\psi\in R$.
\item[(V3)] $\nu(1)=0$ and $\nu(0)=\infty$.
\end{description}
\end{Def}

We denote by $\widetilde{\mathcal{W}}$ the class of all valuations $\nu$ on $R$ such that there exists $\phi\in R$ with $\nu(\phi)\neq 0$ and $\nu(\phi)\neq \infty$.

\begin{Def}
A valuation $\nu:R\lra \Gamma_\infty$ is a \textbf{Krull valuation} if $\nu^{-1}(\infty)=\{0\}$.
\end{Def}

If $R$ admits a Krull valuation $\nu$ then $R$ is a domain and we can extend $\nu$ to a valuation on the field $K=Quot(R)$ by defining $\nu\displaystyle\left(\frac{\phi}{\psi}\right)=\nu\phi-\nu\psi$. The class of all Krull valuations on $R$ (or $K$) will be denoted by $\widetilde{\mathcal{V}}$. By definition $\widetilde{\mathcal{V}}\subseteq \widetilde{\mathcal{W}}$.

Let $\nu: R \lra\Gamma_\infty$ be a valuation. The subgroup of $\Gamma$ generated by
\[
\{\nu(\phi)\mid \phi\in R, \nu(\phi) \neq \infty\}
\]
is called the \textbf{value group of $\nu$} and is denoted by $\nu R$. The valuation is called \textbf{trivial} if $\nu R = \{0\}$. The set $\mathfrak{p}_\nu: = \nu^{-1}(\infty)$ is a prime ideal of $R$, called the \textbf{support of $\nu$}.

Given a valuation $\nu$ on $R$, we can define a Krull valuation
\[
\overline{\nu}:R/\mathfrak{p}_\nu\lra \Gamma_\infty
\]
by defining $\overline{\nu}(\overline{\phi})=\nu\phi$. The local ring
\[
\VR_\nu:=\{\overline{\phi}\in Quot\left(R/\mathfrak{p}_\nu\right)\mid \overline{\nu}\overline{\phi}\geq 0\}
\]
with maximal ideal
\[
\mathfrak{m}_\nu:=\{\overline{\phi}\in Quot\left(R/\mathfrak{p}_\nu\right)\mid \overline{\nu}\overline{\phi}> 0\}
\]
is called the \textbf{valuation ring of $\nu$}.
\begin{Def}
Two valuations $\nu$ and $\mu$ of $R$ are called \textbf{equivalent} ($\nu\sim \mu$) if the following equivalent conditions are satisfied
\begin{description}
\item[i)] For all $\phi, \psi \in R,\ \nu(\phi) > \nu(\psi)$ if and only if $\mu(\phi) >\mu(\psi)$.
\item[ii)] There is an order preserving isomorphism $f:\nu R\lra \mu R$ such that $\mu=f\circ \nu$.
\item[iii)] $\mathfrak{p}_\nu = \mathfrak{p}_\mu$ and $\VR_\nu =\VR_\mu$.
\end{description}
\end{Def}

\begin{Obs}
If $\nu$ and $\mu$ are Krull valuations on $R$, then $\nu\sim \mu$ if and only if $\VR_\nu=\VR_\mu$. Also, if $\nu$ and $\mu$ are two real valued valuations, then $\nu\sim \mu$ if and only if $\nu=C \mu$ for some $C\in \R$ and $C>0$.
\end{Obs}

We denote by $\mathcal{W}$ ($\mathcal{V}$) for the quotient of $\widetilde{\mathcal{W}}$ ($\widetilde{\mathcal{V}}$) by the equivalence relation defined above.

If $R$ is a local ring with maximal ideal $\mathfrak{m}$ we define
\[
\nu_\MI(\phi)=\max\{n\mid \phi\in \MI^n\}.
\]
We will say that a valuation $\nu$ is \textbf{centered} if $\nu(\phi)\geq 0$ for all $\phi\in R$ and $\nu(\phi)>0$ for all $\phi\in \MI$. If $R$ is Noetherian we get that $\MI$ is finitely generated, so we can define
\[
\nu(\MI):=\min \{\nu(\phi)\mid \phi\in \MI\}.
\]
Observe that $\nu_\MI$ is a centered Krull valuation with $\nu_\MI(\MI)=1$.

\begin{Def}\label{Norval}
Let $R$ be a Noetherian local ring with maximal ideal $\MI$. Given an ordered abelian group $\Gamma$ and a positive element $\gamma\in\Gamma$ we say that a valuation $\nu:R\lra\Gamma_\infty$ is \textbf{normalized by $\gamma$} if $\nu(\MI)=\gamma$.
\end{Def}

Fix an ordered abelian group $\Gamma$ and consider the set
\[
\widetilde{\mathcal W}_\Gamma:=\{\nu\in \widetilde{\mathcal W}\mid \nu:R\lra \Gamma_\infty\}.
\]
The family
\[
\left(\{\nu\in\widetilde{\mathcal{W}}_\Gamma\mid \nu\textnormal{ is normalized by $\gamma$ }\}\right)_{\gamma\in\Gamma^{>0}}
\]
forms a partition of $\widetilde{\mathcal{W}}_\Gamma$. If we consider the particular case of $\Gamma=\R$, then every valuation on $\widetilde{\mathcal W}_\R$ is equivalent to a unique valuation normalized by $1$.

Since $\widetilde{\mathcal V}\subseteq \widetilde{\mathcal W}$, we can ask whether there exists a natural subset of $\widetilde{\mathcal W}$ which can be identified with $\widetilde{\mathcal V}$. Would $\widetilde{\mathcal W}_\R$ work? We are looking here for a mapping
\[
\widetilde{\mathcal W}_\R\lra \widetilde{\mathcal V}
\]
which is surjective and injective. Moreover, would this map respect equivalence classes? We present below a mapping which respects equivalence classes, is injective, but not surjective.

Take an element $\nu\in \widetilde{\mathcal{W}}_\R$. If $\mathfrak{p}_\nu=(0)$ then $\nu$ is a Krull valuation and we define $\textrm{krull}[\nu]:=\nu$. If $\mathfrak{p}_\nu\neq(0)$ we have that $(0)\subsetneq \mathfrak{p}_\nu\subsetneq \mathfrak{m}$ which means that $\mathfrak{p}_\nu=(\phi)$ where $\phi\in R$ is an irreducible element. Indeed, $\PI_\nu\neq \MI$ by assumption and if we take any irreducible element $\phi\in \MI$ then
\[
(0)\subsetneq (\phi)\subseteq \PI_\nu\subsetneq \MI
\]
and $(\phi)=\PI_\nu$ because $(\phi)$ is a prime ideal and $\dim(R)=2$. Define now the Krull valuation
\[
\textrm{krull}[\nu]:R\lra \Z\times\R
\]
given by $\textrm{krull}[\nu](\psi)=(r,\nu(\psi'))$ where $\psi=\phi^r\psi'$ and $(\phi,\psi')=1$.

Observe that this definition does not depend on the choice of $\phi$. Indeed, since $R$ is an UFD, any irreducible element $\psi\in \PI_\nu=(\phi)$, would be of the form $u\cdot \phi$ where $u\in R^\times$. It is easy to see that given two valuations $\nu,\mu\in \widetilde{\mathcal W}_\R$, then
\[
\nu\sim \mu\Llr \textrm{krull}[\nu]\sim \textrm{krull}[\mu].
\]
Therefore, the mapping $krull$ induces an injective mapping $\mathcal W_\R\lra \mathcal V$ (which we call again $krull$). We want to study the properties of this mapping.

Take any Krull valuation $\nu:R\lra \Gamma_\infty$. If $rk(\Gamma)=1$ then we can embed $\Gamma$ in $\R$, so there exists a valuation $\nu'\in \widetilde{\mathcal W}_\R$ equivalent to $\nu$. If $rk(\Gamma)=2$ then we can embed $\Gamma$ into $\R\amalg \R$ with the lexicographic order. If the projection of $\nu R$ onto the second coordinate of $\R\amalg \R$ is non-zero, then we consider the valuation given by
\[
\nu'(\phi):=
\begin{cases}
\pi_2(\nu(\phi))&,\text{ if } \pi_1(\nu(\phi))=0\\
\infty&\text{, otherwise. }
\end{cases}
\]
This is a valuation on $R$ and $\textrm{krull}[\nu']\sim\nu$. If the projection of $\nu R$ onto the second coordinate is zero, then there is no valuation $\nu'$ on $R$ such that $\textrm{krull}[\nu']=\nu$. Therefore, the mapping $krull: \mathcal W_\R\lra \mathcal V$ is not surjective.

\begin{Exa}
\begin{description}
Let's give a few examples for the case of $R=\C[[x,y]]$. The monomial valuation on $R$ defined by $\nu(x)=\alpha$ and $\nu(y)=\beta$ is given by
\[
\nu\left(\sum a_{ij}x^iy^j\right)=\min\{i\alpha+j\beta\mid a_{ij}\neq 0\}.
\]
\item[(i)] Take the monomial valuation defined by $\nu(x)=\nu(y)=1$. Then $\nu$ is a Krull valuation and $\nu=\textrm{krull}[\nu]$.

\item[(ii)] Let $\nu$ be the monomial valuation defined by $\nu(x)=1$ and $\nu(y)=\infty$. Then $\PI_\nu=(y)$ and $\textrm{krull}[\nu]$ is the monomial Krull valuation defined by $\nu(x)=(0,1)$ and $\nu(y)=(1,0)$.

\item[(iii)] Consider the Krull monomial valuation $\mu:R\lra \left(\Z\times \Z\right)_\infty$ (with lexicographic order) given by $\mu(x)=(1,0)$ and $\mu(y)=(1,1)$. This is a Krull valuation on $R$ such that there is no valuation $\nu$ on $R$ with $\textrm{krull}[\nu]=\mu$. This shows that $krull: \mathcal W_\R\lra \mathcal V$ is not surjective.
\end{description}
\end{Exa}

\section{The existence of infimum of a set of valuations}\label{Tree}

We will now define rooted non-metric trees and discuss the difference of our way of defining it and the definition given in \cite{Fav_1} and \cite{Gra}.  

\begin{Def}\label{Def_2}
A \textbf{rooted non-metric tree} is a poset $(\mathcal{T},\leq)$ such that:
\begin{description}
\item[(T1)] There exists a (unique) smallest element $\tau_0\in\mathcal{T}$.
\item[(T2)] Every set of the form $I_\tau=\{\sigma\in\mathcal{T}\mid \sigma\leq \tau\}$ is isomorphic (as ordered set) to a real interval.
\item[(T3)] Every totally ordered convex subset of $\mathcal{T}$ is isomorphic to a real interval.
\item[(T4)] Every non-empty subset $\mathcal{S}$ of $\mathcal{T}$ has an infimum in $\mathcal{T}$.
\end{description}
\end{Def}

\begin{Obs}
In \cite{Fav_1} and \cite{Gra} the authors define a rooted non-metric tree without the condition \textbf{(T4)}. In \cite{Fav_1} the authors state that the property \textbf{(T4)} follows from the completeness of the real numbers and the previous properties, which is not true, as the following example shows.
\end{Obs}

\begin{Exa}\label{Exa_1}
Take $X=[0,1)\cup \{x,y\}$ and extend the natural order on $[0,1)$ to $X$ by setting $x,y>[0,1)$ and stating that $x$ and $y$ are incomparable. Then \textbf{(T1)}, \textbf{(T2)} and \textbf{(T3)} hold for $(X,\leq)$, but the set $\{x,y\}$ does not have an infimum.
\end{Exa}

\begin{Lema}
Under the conditions \textbf{(T1)} and \textbf{(T2)}, the following conditions are equivalent:
\begin{description}
\item[(T4)] Every non-empty subset $\mathcal S\subseteq \mathcal{T}$ has an infimum.
\item[(T4')] Given two elements $\tau,\sigma\in\mathcal{T}$, the set $\{\tau,\sigma\}$ has an infimum $\tau\wedge\sigma$.
\end{description}
\end{Lema}
\begin{proof}
\textbf{(T4')} is a particular case of \textbf{(T4)}. Assume now that \textbf{(T4')} holds and take $\mathcal S\subseteq \mathcal{T}$ a non-empty subset. We have to prove that $\mathcal S$ has an infimum. Fix an element $\tau\in \mathcal S$ and let
\[
\Phi_\tau:[\tau_0,\tau]\lra [a,b]\subseteq \R
\]
be the isomorphism given by property \textbf{(T2)}. For each $\sigma\in \mathcal S$, by property \textbf{(T4')}, there exists an element $\tau\wedge\sigma\in \mathcal T$ which is the infimum of $\{\tau,\sigma\}$ in $\mathcal T$. Define the element
\[
a_{\sigma}=\Phi_\tau(\tau\wedge\sigma)\in [a,b].
\]
Since $\R$ is complete, we have that $\{a_{\sigma}\mid \sigma\in \mathcal S\}$ has an infimum $a_0\in [a,b]$. Define the element $\sigma_0=\Phi_\tau^{-1}(a_0)\in\mathcal T$. Let's prove that $\sigma_0=\inf \mathcal S$. If not, there would be an element $\sigma_0'\in \mathcal{T}$ such that $\sigma_0<\sigma_0'\leq \sigma$ for all $\sigma\in \mathcal S$. Then $\sigma'_0\leq \tau\wedge\sigma$  and hence $a_0=\Phi_\tau(\sigma_0)<\Phi_\tau(\sigma'_0)\leq a_\sigma$ for all $\sigma\in \mathcal S$, which shows that $a_0<\inf\{ a_\sigma\mid \sigma\in \mathcal S\}$, a contradiction.
\end{proof} 

We will now define some properties associated with a non-metric tree.
\begin{Def}\label{Def_1}
\begin{description}

\item[(i)] Given a non-empty subset $\mathcal{S}\subseteq\mathcal{T}$ we define the \textbf{join} $\displaystyle\bigwedge_{\tau\in\mathcal{S}}\tau$ of $\mathcal{S}$ to be the infimum of $\mathcal{S}$.

\item[(ii)] Given two elements $\tau,\sigma\in\mathcal{T}$ we define the \textbf{closed segment} connecting them by
\[
\left[\tau,\sigma\right]:=\{\alpha\in\mathcal{T}\mid (\tau\wedge\sigma\leq\alpha\leq \tau)\vee (\tau\wedge\sigma\leq\alpha\leq \sigma)\}.
\]
We define similarly $\left]\tau,\sigma\right]$ and $\left[\tau,\sigma\right[$.

\item[(iii)] Take a point $\tau\in\mathcal{T}$ and define an equivalence relation on $\mathcal{T}\backslash \{\tau\}$ by setting
\[
\sigma\sim_{\tau}\alpha\Llr\left]\tau,\sigma\right]\cap\left]\tau,\alpha\right]\neq\emptyset.
\]
The \textbf{tangent space of $\mathcal{T}$ at $\tau$} is the set of equivalence classes of $\mathcal{T}\backslash \{\tau\}$. An equivalence class $[\sigma]_{\tau}\in\mathcal{T}\backslash \{\tau\}/\sim_{\tau}$ is called a \textbf{tangent vector at $\tau$}.

\item[(iv)] The \textbf{weak tree topology} on $\mathcal{T}$ is the topology generated by the tangent vectors at points of $\mathcal{T}$, i.e., the open sets are unions of finite intersections of sets of the form $[\sigma]_{\tau}$.

\item[(v)] A \textbf{parametrization} of a rooted non-metric tree is an increasing (or decreasing) mapping $\Psi:\mathcal{T}\lra\left[-\infty,\infty\right]$ such that its restriction to every totally ordered convex subset of $\mathcal{T}$ is an isomorphism (of ordered sets) onto a real interval.

\item[(vi)] Given a (increasing) parametrization $\Psi:\mathcal{T}\lra\left[1,\infty\right]$ we define a metric on $\mathcal{T}$ by setting
\[
d_{\Psi}(\tau,\sigma)=\displaystyle\left(\frac{1}{\Psi(\tau\wedge\sigma)}-\frac{1}{\Psi(\tau)}\right)+\left(\frac{1}{\Psi(\tau\wedge\sigma)}-\frac{1}{\Psi(\sigma)}\right).
\]
\end{description}
\end{Def}

\begin{Obs}
Observe that the definitions above depend strongly on the existence of an infimum for any two given elements.
\end{Obs}

We will start to prove now that every pair of valuations centered at a two-dimensional regular local ring $(R,\MI)$ admits an infimum.

Let $(R,\MI)$ be a local ring and consider the set
\[
\mathcal S=\left(R^\times\cup\{0\}\right)^2\setminus \{(0,0)\}.
\]

Define a relation on $\mathcal S$ by setting
\[
(a_1,b_1)\sim (a_2,b_2)\Llr a_1b_2-a_2b_1\in \MI.
\]

\begin{Lema}
This relation is an equivalence relation.
\end{Lema}
\begin{proof}
This relation is clearly reflexive and symmetric, so it remains to show that it is transitive. Suppose that $(a_1,b_1)\sim (a_2,b_2)$ and $(a_2,b_2)\sim (a_3,b_3)$. By definition, we have that $a_1b_2-a_2b_1,a_2b_3-a_3b_2\in \MI$. If $a_2\neq 0$ then we have that
\[
a_2(a_3b_1-a_1b_3)=a_3a_2b_1-a_3a_1b_2+a_1a_3b_2-a_1a_2b_3=a_3(a_2b_1-a_1b_2)+a_1(a_3b_2-a_2b_3)\in \MI
\]
and since $a_2\in R^\times$ we have that $a_3b_1-a_1b_3\in \MI$. If $a_2=0$ then $b_2\neq 0$ and we have
\[
b_2(a_3b_1-a_1b_3)=b_1a_3b_2-b_1a_2b_3+b_3a_2b_1-b_3a_1b_2=b_1(a_3b_2-a_2b_3)+b_3(a_2b_1-a_1b_2)\in \MI
\]
and again $a_3b_1-a_1b_3\in \MI$.
\end{proof}

Suppose that $(R,\MI)$ is a two-dimensional regular local ring and let $(x,y)$ be a regular system of parameters of $\mathfrak{m}$. Take a valuation $\nu$ centered at $R$.

\begin{Lema}
Take $(a_1,b_1),(a_2,b_2)\in \mathcal S$ with $(a_1,b_1)\sim (a_2,b_2)$. Then
\[
\nu(a_1x+b_1y)>\nu(\MI)\Llr\nu(a_2x+b_2y)>\nu(\MI).
\]
Also, if there exist $(a_1,b_1),(a_2,b_2)\in \mathcal S$ such that $\nu(a_1x+b_1y)>\nu(\MI)$ and $\nu(a_2x+b_2y)>\nu(\MI)$ then $(a_1,b_1)\sim (a_2,b_2)$.
\end{Lema}

\begin{proof}
For the first statement, suppose that $a_1\neq 0$ (and so $a_2\neq 0$). Then
\begin{displaymath}
\begin{array}{rl}
a_2x+b_2y & =\displaystyle\frac{a_2}{a_1}\left(a_1x+\frac{a_1b_2}{a_2}y\right)\\
          & =\displaystyle\frac{a_2}{a_1}\left(a_1x+b_1y-b_1y+\frac{a_1b_2}{a_2}y\right)\\
          & =\displaystyle\frac{a_2}{a_1}\left(a_1x+b_1y\right)+\frac{a_1b_2-a_2b_1}{a_1}y
          
\end{array}
\end{displaymath}
Since $(a_1,b_1)\sim (a_2,b_2)$ we have that
\[
\nu\left(\frac{a_1b_2-a_2b_1}{a_1}y\right)=\nu(a_1b_2-a_2b_1)-\nu(a_1)+\nu(y)>\nu(y)\geq \nu(\MI).
\]

If $\nu(a_1x+b_1y)>\nu(\MI)$ then
\[
\nu(a_2x+b_2y)\geq \min\left\{\nu\left(\frac{a_2}{a_1}\left(a_1x+b_1y\right)\right),\nu\left(\frac{a_1b_2-a_2b_1}{a_1}y\right)\right\}>\nu(\MI),
\]
and if $\nu(a_1x+b_1y)=\nu(\MI)<\nu\left(\displaystyle\frac{a_1b_2-a_2b_1}{a_1}y\right)$ then
\[
\nu(a_2x+b_2y)= \min\left\{\nu\left(\frac{a_2}{a_1}\left(a_1x+b_1y\right)\right),\nu\left(\frac{a_1b_2-a_2b_1}{a_1}y\right)\right\}=\nu(\MI).
\]
If $a_1=0$ then $b_1\neq 0\neq b_2$ and we proceed similarly to get
\[
a_2x+b_2y=\displaystyle\frac{b_2}{b_1}\left(a_1x+b_1y\right)+\frac{a_2b_1-a_1b_2}{b_1}x.
\]

For the second statement, suppose that
\[
\nu(a_1x+b_1y)>\nu(\MI)\textnormal{ and }\nu(a_2x+b_2 y)>\nu(\MI)
\]
and that $(a_1,b_1)\nsim (a_2,b_2)$. This would mean that $a_1b_2-a_2b_1\notin \MI$, so $\nu(a_2b_1-a_1b_2)=0$. Then we would have that 
\begin{displaymath}
\begin{array}{rcl}
\nu(x)&=&\nu(a_2b_1x-a_1b_2x)\\
    &=&\nu(a_2b_1x+b_1b_2y-b_1b_2y-a_1b_2x)\\
    &=&\nu(b_1(a_2x+b_2y)-b_2(a_1x+b_1y))\\
    &>&\nu(\MI)
\end{array}
\end{displaymath}
and
\begin{displaymath}
\begin{array}{rcl}
\nu(y)&=&\nu(a_2b_1y-a_1b_2y)\\
    &=&\nu(a_2b_1y+a_1a_2x-a_1a_2x-a_1b_2y)\\
    &=&\nu(a_2(a_1x+b_1y)-a_1(a_2x+b_2y))\\
    &>&\nu(\MI)
\end{array}
\end{displaymath}
which is a contradiction to $\nu(\mathfrak{m})=\min\{\nu(x),\nu(y)\}$.
\end{proof}
\begin{Def}
Take an element $\lambda\in \mathcal S/\sim$. We say that 
\[
\nu(x+\lambda y)>\nu(\MI)
\]
if $\nu(a_1x+b_1y)>\nu(\MI)$ for some (and hence for every) $(a_1,b_1)\in \lambda$. Analogously, we say that 
\[
\nu(x+\lambda y)=\nu(\MI)
\]
if $\nu(a_1x+b_1y)=\nu(\MI)$ for some (and hence for every) $(a_1,b_1)\in \lambda$.
\end{Def}

\begin{Cor}\label{Cor_1}
Given two centered valuations $\nu,\mu:R\lra \R_\infty$, there exist elements $a,b\in R^\times\cup\{0\}$ such that $\nu(\MI)=\nu(ax+by)$ and $\mu(\MI)=\mu(ax+by)$.
\end{Cor}
\begin{proof}
By the second part of the Lemma above, there exist at most one $\lambda_\nu\in \mathcal S/\sim$ and at most one $\lambda_\mu\in \mathcal S/\sim$ such that $\nu(x+\lambda_\nu y)>\nu(\MI)$ and $\mu(x+\lambda_\mu y)=\mu(\MI)$. Since $|\mathcal S/\sim|\geq 3$ for any ring $R$ there exists an element $\lambda\in \mathcal S/\sim$ with $\lambda_\nu\neq\lambda\neq\lambda_\mu$. Take any $(a,b)\in \lambda$ and we have that $\nu(ax+by)=\nu(\MI)$ and $\mu(ax+by)=\mu(\MI)$.
\end{proof}

\begin{Cor}
If $R=k[[x,y]]$ for an algebraically closed field $k$ and $\nu\neq\nu(\MI)\cdot\nu_\MI$ then there exists $\lambda\in \p^1(k)$ such that $\nu(x+\lambda y)>\nu(\MI)$.
\end{Cor}
\begin{proof}
We will prove that if $R=k[[x,y]]$ and $\nu\neq\nu(\MI)\cdot\nu_\MI$ where $k$ is any field, then there exists a homogeneous polynomial $m$ in $(x,y)$ such that $\nu(m)>\nu_\MI(m)\cdot\nu(\MI)$. Consequently, if $k$ is algebraically closed then we can find a homogeneous polynomial $m$ of degree $1$ such that $\nu(m)>\nu_\MI(m)\cdot\nu(\MI)=\deg (m)\cdot\nu(\MI)=\nu(\MI)$.

Since $\nu\neq \nu_\MI\cdot\nu(\MI)$ there exists a power series
\[
p(x,y)=\sum_{i=1}^\infty m_i \in k[[x,y]]\textnormal{ where }m_i\textnormal{ are monomials in }x\textnormal{ and }y,
\]
such that $\nu(p)>\nu_\MI(p)\cdot\nu(\MI)=k\cdot \nu(\MI)$ where $k:=ord_\MI(p)$. Write $p$ as sum of homogeneous polynomials, i.e.,
\[
p(x,y)=\sum_{j\geq k} p_j,\textnormal{ where }p_j=\displaystyle\sum_{\deg(m_i)=j}m_i.
\]
Since $\nu(p_j)\geq \min\{\nu(m_i)\mid \deg(m_i)=j\}\geq j\cdot \nu(\MI)$ and $\nu(p-p_k)>k\cdot\nu(\MI)$ we have that $\nu(p_k)> k\cdot\nu(\MI)$. Indeed, if $\nu(p_k)=k\cdot\nu(\MI)$ we would have that
\[
\nu(p)=\min\{\nu(p-p_k),\nu(p_k)\}=\nu(p_k)=k\cdot\nu(\MI)=\nu_\MI(p)\cdot\nu(\MI).
\]
Therefore, there exists a homogeneous polynomial $m\in k[[x,y]]$ (namely $m=p_k$) such that $\nu(m)>\nu_\MI(m)\cdot\nu(\MI)$. If $k$ is algebraically closed then $m$ can be chosen to be of degree one.
\end{proof}

\begin{Obs}
In \cite{Fav_1} Corollary 3.19, page 48, it is proved that if two valuations $\nu$ and $\mu$ on $\C[[x,y]]$ are given where $(x,y)$ are coordinates such that $\nu(x)=\mu(x)=1\leq \min\{\nu(y),\mu(y)\}$ then there exists the infimum for $\nu$ and $\mu$. By the Corollary above, we conclude that every pair of valuations on $\C[[x,y]]$ have an infimum.
\end{Obs}

For each valuation $\nu$ centered at $R$ take a regular system of parameters $(x,y)$ such that $\nu(x)\leq \nu(y)$. Let $\nu'$ be the unique extension of $\nu$ to $R\left[\displaystyle\frac yx\right]$ with $\nu'\left(\displaystyle\frac yx\right) =\nu(y)-\nu(x)$ and let
\[
\mathfrak q^{(1)}_\nu=\left\{r\in R\left[\displaystyle\frac yx\right]\mid \nu'(r)>0\right\}.
\]
Then the local ring
\[
R_\nu^{(1)}=R\left[\frac yx\right]_{\mathfrak{q}_\nu^{(1)}}
\]
is called the \textbf{quadratic dilatation of $R$ with respect to $\nu$}. If $\dim R_\nu^{(1)}=1$ then $\nu=a\cdot \nu_{\MI}$ for some $a>0$. If $\dim R_\nu^{(1)}=2$ then we proceed as before to get a new local ring $R_\nu^{(2)}$ which is the quadratic dilatation of $R_\nu^{(1)}$ with respect to $\nu^{(1)}$. We can construct inductively a sequence (finite or infinite)
\[
R\subseteq R_\nu^{(1)}\subseteq R_\nu^{(2)}\subseteq \ldots\subseteq R_{\nu}^{(n)}\subseteq \ldots
\]
of regular local rings such that $R_\nu^{(i)}$ is the quadratic dilatation of $R_\nu^{(i-1)}$ with respect to $\nu^{(i-1)}$ (here $R_\nu^{(0)}:=R$). Let $\lambda(\nu)$ be the length of the sequence above, i.e.,
\[
\lambda(\nu)=\begin{cases}n+1&\text{, if }\dim \left(R_{\nu}^{(n)}\right)=2\text{ and }\dim \left(R_{\nu}^{(n+1)}\right)=1\\
        \infty&\text{, if }\dim \left(R_{\nu}^{(n)}\right)=2\text{ for every }n\in \N.\end{cases}
\]
It is proved in \cite{Ab_1} that
\[
\VR_{\textrm{krull}[\nu]}=\bigcup_{i=0}^{\lambda(\nu)} R_{\nu}^{(i)}
\]
where $\VR_{\textrm{krull}[\nu]}$ is the valuation ring of $\textrm{krull}[\nu]$ in $Quot(R)$.
The sequence $\left\{R_\nu^{(i)}\right\}_{i=0}^{\lambda(\nu)}$ is called the sequence of quadratic dilatations of $\nu$ and the sequence $\left\{m_\nu^{(i)}\right\}_{i=0}^{\lambda(\nu)}$ where $m_\nu^{(i)}=\nu^{(i)}\left(\MI^{(i)}_\nu\right)$ is called the \textbf{multiplicity sequence of $\nu$}.

Fix a regular system of parameters $(x,y)$ of $\MI$. For each $\phi\in R\setminus \{0\}$ there exists a unique decomposition $\phi=a_1M_1+\ldots+a_nM_n$ where $a_i\in R\setminus \MI$ and $M_i=x^{r_i}y^{s_i}$ is a pure monomial in $(x,y)$, $0\leq i\leq n$. Take $\gamma_1,\gamma_2\in\R_\infty$ not both equal to $\infty$. Then we define the valuation
\[
\nu(\phi)=\min_{1\leq i\leq n}\{r_i\gamma_1+s_i\gamma_2\}.
\]
This is indeed a valuation (see Lemma 7 of \cite{Gra}) and it is called a monomial valuation in $(x,y)$. It is a Krull valuation if $\gamma_1\neq\infty\neq\gamma_2$ and it is centered if $\gamma_1>0$ and $\gamma_2>0$ (see Lemma 8 of \cite{Gra}).

To prove Theorem \ref{Main} we will use the following Theorem (Theorem 18 of \cite{Gra}):

\begin{Teo}\label{Lem_1}
Let $\nu$ and $\mu$ be two centered valuations of $R$ and suppose that $\mu\leq\nu$. Assume that there exists $s\geq 0$ such that $\dim\left(R_\mu^{(i)}\right)=2$, $m_\nu^{(i)}=m_\mu^{(i)}$ for $0\leq i\leq s$ and either $\dim\left(R_\nu^{(s+1)}\right)=1$ or $\dim\left(R_\nu^{(s+1)}\right)=2$ and $m^{(s+1)}_\nu>m^{(s+1)}_\mu$. Then $R^{(i)}_\nu=R^{(i)}_\mu$ and $0\leq\mu^{(i)}(\phi)\leq\nu^{(i)}(\phi)$ for each $\phi\in R^{(i)}$, $0\leq i\leq s$. Moreover, we have the following possibilities:
\begin{description}
\item[(a)] If $\dim \left(R_\nu^{(s+1)}\right)=1$, then $\lambda(\nu)=\lambda(\mu)=s+1$ and $\nu^{(s)}=\mu^{(s)}=m^{(s)}_\nu\cdot \nu_{\MI^{(s)}}$.

\item[(b)] If $\dim \left(R_\nu^{(s+1)}\right)=2$ and $\dim \left(R_\mu^{(s+1)}\right)=1$, then $s+1=\lambda(\mu)<\lambda(\nu)$ and there exists a monomial valuation $\overline{\mu}^{(s+1)}$ on $R^{(s+1)}_\nu$ such that $\mu^{(s)}$ is the restriction of $\overline{\mu}^{(s+1)}$ to $R^{(s)}_\nu$.

\item[(c)] If $\dim \left(R_\nu^{(s+1)}\right)=2$ and $\dim \left(R_\mu^{(s+1)}\right)=2$, then $s+1<\min\{\lambda(\nu),\lambda(\mu)\}$, $R^{(s+1)}_\nu=R^{(s+1)}_\mu$, $0\leq \mu^{(s+1)}(\phi)\leq \nu^{(s+1)}(\phi)$ for all $\phi\in R^{(s+1)}$ and $\mu^{(s+1)}$ is a monomial valuation on $R^{(s+1)}$.
\end{description}
\end{Teo}

\begin{proof}[Proof of Theorem \ref{Main}]
Take two centered valuations $\nu,\mu:R\lra \R_\infty$ such that $\nu(\MI)=\mu(\MI)=1$. Since $R^{(0)}_\nu=R=R^{(0)}_\mu$ and $1=\nu(\MI)=m^{(0)}_\nu=m^{(0)}_\mu$ we can define
\[
s=\max\{i\mid R^{(i)}_\nu=R^{(i)}_\mu\text{ and }m^{(i)}_\nu=m^{(i)}_\mu\}.
\]
If $s=\infty$ then $\VR_{\textrm{krull}[\nu]}=\VR_{\textrm{krull}[\mu]}$ and consequently $\nu\sim\mu$. Since these valuations are normalized by $\nu(\MI)=1=\mu(\MI)$ we must have that $\nu=\mu$ and there is nothing to prove. Therefore, assume that $s<\infty$.
We define $R^{(i)}:=R^{(i)}_\nu=R^{(i)}_\mu$ and $m^{(i)}:=m^{(i)}_\nu=m^{(i)}_\mu$ for $0\leq i\leq s$.

We will divide our proof in cases, starting with the case where $R^{(s+1)}_\nu\neq R^{(s+1)}_\mu$. By Corollary \ref{Cor_1} there exists $x^{(s)}\in\MI^{(s)}$ such that
\[
\nu^{(s)}\left(x^{(s)}\right)=\nu^{(s)}\left(\MI^{(s)}\right)=\mu^{(s)}\left(\MI^{(s)}\right)=\mu^{(s)}\left(x^{(s)}\right).
\]
Take any $y^{(s)}\in\MI^{(s)}$ such that $\left(x^{(s)},y^{(s)}\right)$ is a regular system of parameters for $\MI^{(s)}$. Define $\omega^{(s)}$ to be the monomial valuation on $R^{(s)}$ defined by
\[
\omega^{(s)}\left(x^{(s)}\right)=\nu^{(s)}\left(x^{(s)}\right)=\mu^{(s)}\left(x^{(s)}\right)\text{ and }\omega^{(s)}\left(y^{(s)}\right)=\min\{\nu\left(y^{(s)}\right),\mu\left(y^{(s)}\right)\}.
\]

Let $\omega$ be the restriction of $\omega^{(s)}$ to $R$. From the definition of monomial valuation, we conclude that $\omega\leq \nu$ and $\omega\leq\mu$. We want to prove that if $\omega'$ is a valuation of $R$ such that $\omega\leq\omega'\leq\nu$ and $\omega\leq\omega'\leq\mu$, then $\omega=\omega'$.

If $\dim\left(R^{(s+1)}_{\omega'}\right)=1$, applying Theorem \ref{Lem_1} \textbf{(a)} for $\omega\leq \omega'$ we have that $\omega=\omega'$. If $\dim \left(R^{(s+1)}_{\omega'}\right)=2$ then $\dim \left(R^{(s+1)}_{\nu}\right)=2$ and $\dim \left(R^{(s+1)}_{\mu}\right)=2$. Moreover, applying Theorem \ref{Lem_1} \textbf{(c)} for $\omega'\leq \nu$ and $\omega'\leq \mu$ we have that $R^{(s+1)}_\mu=R^{(s+1)}_{\omega'}$ and $R^{(s+1)}_\nu=R^{(s+1)}_{\omega'}$. Consequently, $R^{(s+1)}_\nu=R^{(s+1)}_\mu$, which is a contradiction with our assumption. Therefore, $\dim\left(R^{(s+1)}_{\omega'}\right)=1$ and $\omega=\omega'$.

The remaining case is if $R^{(s+1)}_\nu= R^{(s+1)}_\mu=:R^{(s+1)}$ and $m_\nu^{(s+1)}\neq m_\mu^{(s+1)}$, say $m_\nu^{(s+1)}< m_\mu^{(s+1)}$. Define the valuation $\omega^{(s+1)}$ in $R^{(s+1)}_\nu$ to be the monomial valuation given by
\[
\omega^{(s+1)}\left(x^{(s+1)}\right)=\min\left\{\nu^{(s+1)}\left(x^{(s+1)}\right),\mu^{(s+1)}\left(x^{(s+1)}\right)\right\}
\]
and
\[
\omega^{(s+1)}\left(y^{(s+1)}\right)=\min\left\{\nu^{(s+1)}\left(y^{(s+1)}\right),\mu^{(s+1)}\left(y^{(s+1)}\right)\right\},
\]
where $\left(x^{(s+1)},y^{(s+1)}\right)$ is a regular system of parameters for $R^{(s+1)}$ with the property that $\nu^{(s+1)}\left(x^{(s+1)}\right)=\nu^{(s+1)}\left(\MI^{(s+1)}\right)$ and $\mu^{(s+1)}\left(x^{(s+1)}\right)=\mu^{(s+1)}\left(\MI^{(s+1)}\right)$ (such $x^{(s+1)}$ exists by Corollary \ref{Cor_1}). Let $\omega$ be the restriction of $\omega^{(s+1)}$ to $R$. Take a valuation $\omega'$ in $R$ such that $\omega\leq\omega'\leq \nu$ and $\omega\leq\omega'\leq\mu$. We want to prove that $\omega=\omega'$.

From our definition, we get that $m_\omega^{(s+1)}=m_\nu^{(s+1)}<m_\mu^{(s+1)}$. If $m_{\omega'}^{(s+1)}>m_\omega^{(s+1)}$ then we would have that $\omega'\nleq \nu$ which is a contradiction. Thus $m_{\omega'}^{(s+1)}=m_\omega^{(s+1)}$. Since $\omega'\leq\mu$ and $m_\mu^{(s+1)}>m_{\omega'}^{(s+1)}$ we are in the situation of the Theorem \ref{Lem_1}, so by \textbf{(c)} we have that $\omega'^{(s+1)}$ is monomial (on $\left(x^{(s+1)},y^{(s+1)}\right)$). Therefore, $\omega'^{(s+1)}=\omega^{(s+1)}$ and consequently $\omega=\omega'$.
\end{proof}
\begin{Obs}
In the proof, we used the fact that in the situation above, the valuation $\omega'^{(s+1)}$ is a monomial valuation on the coordinates $\left(x^{(s+1)},y^{(s+1)}\right)$. This fact was not explicitly stated but appears in the proof of Theorem \ref{Lem_1} in \cite{Gra}.
\end{Obs}

\section{Comparison of topologies}\label{Top}

In this section we compare the topologies defined above with classical topologies. Also, we compare the weak tree topology and the metric topology given by a parametrization of a rooted non-metric tree.

The most well known topology on space of Krull valuations is the Zariski Topology:

\begin{Def}
The Zariski topology on $\mathcal{V}$ is the topology generated by the sets of the form
\[
\{[\nu]\in \mathcal{V}\mid \phi\in \mathcal{O}_\nu\}
\]
where $\phi\in K$.
\end{Def}

\begin{Obs}
It is proved in \cite{Zar} that this topology is compact (quasi-compact) but not Hausdorff.
\end{Obs}

The coarsest Hausdorff topology which is finer than the Zariski topology on $\mathcal{V}$ is the Patch-Zariski topology:
\begin{Def}
The Patch-Zariski topology on $\mathcal{V}$ is defined to be the topology generated by the sets of form
\[
\{[\nu]\in \mathcal{V}\mid \phi\in \mathcal{O}_\nu\}
\]
and
\[
\{[\nu]\in \mathcal{V}\mid \psi\in \MI_\nu\}
\]
where $\phi,\psi\in K$.
\end{Def}

\begin{Obs}
It is proved in \cite{Kuh_1} that the Zariski topology is spectral and that the patch topology associated to it is indeed the Patch-Zariski topology defined above.
\end{Obs}

We will describe below an approach used by Berkovich in \cite{Ber} and by Favre in \cite{Fav_1} and \cite{Fav_2} to define topologies on sets of valuations.

\begin{Def}
Consider the subset $\mathcal W_\MI$ of $\widetilde{\mathcal{W}}_\R$ consisting of all valuations normalized by $1$. The set $\mathcal W_\MI$ is called the valuative tree of $R$. We define the weak topology on $\mathcal W_\MI$ to be the topology generated by the sets of the form
\[
\{\nu\in \mathcal W_\MI\mid \nu(\phi)>\alpha\}\textnormal{ and }\{\nu\in \mathcal W_\MI\mid \nu(\phi)<\alpha\}
\]
where $\alpha\in \R$ and $\phi\in R$.
\end{Def}

\begin{Obs}
It is easy to see that $\mathcal W_\MI\subseteq \left(\R_\infty\right)^{R}$ and that the topology defined above is the topology on $\mathcal W_\MI$ induced by the product topology on $\left(\R_\infty\right)^{R}$ where $\R_\infty$ is regarded with the order topology.
\end{Obs}

An interesting fact, proved in \cite{Fav_1} (Theorem 5.1) is the following:
\begin{Prop}
The weak tree topology and the weak topology in $\mathcal W_\MI$ are the same.
\end{Prop}

We will proceed with the proof of Theorem \ref{Comp}.
\begin{proof}[Proof of Theorem \ref{Comp}]
It is enough to show that every subbasic set $[\sigma]_{\tau}$ in the weak tree topology is open in the metric topology, i.e., for every $\gamma\in[\sigma]_{\tau}$ there exist $\epsilon>0$ such that
\[
B_\epsilon(\gamma)=\{\alpha\in \mathcal T\mid d_\Psi(\gamma,\alpha)<\epsilon\}\subseteq [\sigma]_{\tau}.
\]

By definition $\gamma\neq\tau$ so $\epsilon:=d_\Psi(\gamma,\tau)>0$. Let's prove that $B_\epsilon(\gamma)\subseteq [\sigma]_{\tau}$.
\begin{Claim}\label{Claim_1}
$\alpha\notin [\sigma]_{\tau}\Llr \tau\in\left[\alpha,\sigma\right]$.
\end{Claim}
\begin{proof}

Suppose first that $\alpha\notin [\sigma]_{\tau}$. This means that $\left]\tau,\alpha\right]\cap\left]\tau,\sigma\right]=\emptyset$. Suppose towards a contradiction that $\tau>\alpha,\sigma$. Then we would have that $\alpha$ and $\sigma$ are comparable, say $\alpha\leq\sigma$. This means that $\left]\tau,\sigma\right]\subseteq\left]\tau,\alpha\right]$ which is a contradiction. Consequently, $\tau<\sigma$ or $\tau<\alpha$. It remains to show that $\tau\geq\alpha\wedge\sigma$. Suppose not, i.e., that $\tau<\alpha\wedge\sigma$. Then we would have that $\alpha\wedge\sigma\in\left]\tau,\alpha\right]\cap\left]\tau,\sigma\right]$ which is again a contradiction.

For the converse, assume that $\tau\in\left[\alpha,\sigma\right]$. If $\tau=\alpha\wedge\sigma$ then we get by definition of $\alpha\wedge\sigma$ that $\left]\tau,\alpha\right]\cap\left]\tau,\sigma\right]=\emptyset$. Otherwise, assume w.l.o.g. that $\alpha\wedge\sigma<\tau<\alpha$. Then
\[
\left]\tau,\alpha\right]=\{\gamma\in\mathcal{T}\mid\tau<\gamma\leq\alpha\}
\]
and
\[
\left]\tau,\sigma\right]=\{\gamma\in\mathcal{T}\mid(\alpha\wedge\sigma\leq\gamma<\tau)\vee(\alpha\wedge\sigma\leq\gamma\leq\sigma)\}
\]
which are disjoint sets. Therefore, $\alpha\nsim_\tau \sigma$, so $\alpha \notin[\sigma]_{\tau}$.
\end{proof}
\begin{Claim}\label{Claim_2}
If $\tau\in\left[\alpha,\sigma\right]$ then $d_\Psi(\alpha,\sigma)=d_\Psi(\alpha,\tau)+d_\Psi(\tau,\sigma)$
\end{Claim}
\begin{proof}
Suppose w.l.o.g that $\alpha\wedge\sigma\leq \tau\leq \sigma$. Then we have that $\alpha\wedge\tau=\alpha\wedge\sigma$ and that $\tau\wedge\sigma=\tau$. Therefore,
\begin{displaymath}
\begin{array}{rcl}
d_\Psi(\alpha,\sigma) & = & \displaystyle\left(\frac{1}{\Psi(\alpha\wedge\sigma)}-\frac{1}{\Psi(\alpha)}\right)+\left(\frac{1}{\Psi(\alpha\wedge\sigma)}-\frac{1}{\Psi(\sigma)}\right)\\
                              & = & \displaystyle\left(\frac{1}{\Psi(\alpha\wedge\tau)}-\frac{1}{\Psi(\alpha)}\right)+\left(\frac{1}{\Psi(\alpha\wedge\tau)}-\frac{1}{\Psi(\sigma)}\right)+\left(\frac{2}{\Psi(\tau\wedge\sigma)}-\frac{2}{\Psi(\tau)}\right)\\
                              & = & \displaystyle\left(\frac{1}{\Psi(\alpha\wedge\tau)}-\frac{1}{\Psi(\alpha)}\right)+\left(\frac{1}{\Psi(\alpha\wedge\tau)}-\frac{1}{\Psi(\tau)}\right)\\
                              &   & + \displaystyle\left(\frac{1}{\Psi(\tau\wedge\sigma)}-\frac{1}{\Psi(\sigma)}\right)+\left(\frac{1}{\Psi(\tau\wedge\sigma)}-\frac{1}{\Psi(\tau)}\right)\\\\
                              & = & d_\Psi(\alpha,\tau)+d_\Psi(\tau,\sigma).
                              
\end{array}
\end{displaymath}
\end{proof}

Take an element $\alpha\notin[\sigma]_{\tau}=[\gamma]_{\tau}$. By Claim \ref{Claim_1} we have that $\tau\in\left[\alpha,\gamma\right]$. By Claim \ref{Claim_2} we have that
\[
d_\Psi(\alpha,\gamma)=d_\Psi(\alpha,\tau)+d_\Psi(\tau,\gamma)=d_\Psi(\alpha,\tau)+\epsilon\geq\epsilon.
\]
Therefore, $\alpha\notin B_\epsilon(\gamma)$ and consequently, $B_\epsilon(\gamma)\subseteq [\sigma]_{\tau}$.
\end{proof}

We now analyse if these topologies are the same:

\begin{Teo}\label{Teo_1}
If there is an element $\sigma\in\mathcal{T}$ with uncountably many branches $(|\mathcal{T}_{\sigma}|>|\N|)$ then the weak tree topology is not first countable. In particular, the metric topology given by any parametrization is strictly coarser than the weak tree topology.
\end{Teo}
\begin{proof}
Take an element $\sigma\in\mathcal{T}$ which has uncountably many branches. Observe that $[\sigma]_{\tau}$ contains all branches emanating from $\sigma$ except for the one on which $\tau$ lies. That means that any basic open set (therefore any open set) that contains $\sigma$ contains uncountably many branches emanating from $\sigma$. Take now any family $\{V_n\}_{n\in\N}$ of open sets containing $\sigma$. Since each $V_n$ contains uncountably many branches, their intersection contains uncountably many branches. Take one of these branches and choose an element $\alpha$ on it. Take now the subbasic open set $[\sigma]_\alpha$. Then $\alpha\in V_n$ for all $n\in\N$ and $\alpha\notin[\sigma]_\alpha$, so $V_n\nsubseteq [\sigma]_\alpha$. Therefore, there is no countable system of neighbourhoods for the element $\sigma$.
\end{proof}

\begin{Cor}
In the valuative tree, the weak tree topology is strictly coarser than the topology generated by a parametrization.
\end{Cor}
\begin{proof}
It is proved in \cite{Fav_1} that divisorial valuations have uncountably many branches. By the theorem above we get that the weak tree topology and the metric topology defined by a parametrization are different.
\end{proof}
\begin{Obs}
As a criterion for the topologies to be equal or different, the fact that there exists a point with uncountably many branches is the best that we can get. We can present examples of trees in which every point has finitely (or countably) many branches and the topologies are equal and examples where they are different.
\end{Obs}

\smallskip

\noindent
\bf Acknowledgement. \rm The author would like to thank to his supervisor Franz-Viktor Kuhlmann for his valuable comments and corrections.

\end{document}